\documentclass{daj}
\usepackage{amsmath,amsthm,amssymb,color,verbatim,graphicx,fullpage,url,cite,cleveref}

\dajAUTHORdetails{%
  title = {The Analytic Rank of Tensors and Its Applications}, 
  author = {Shachar Lovett},
  plaintextauthor = {Shachar Lovett},
  keywords = {Tensors, Partition rank, Analytic rank},
}   

\dajEDITORdetails{%
   year={2019},
   number={7},
   received={9 October 2018},   
   published={3 June 2019},  
   doi={10.19086/da.8654},       
}   

\newtheorem{theorem}{Theorem}[section]
\newtheorem{claim}[theorem]{Claim}
\newtheorem{lemma}[theorem]{Lemma}

\newtheorem{example}[theorem]{Example}
\newtheorem{problem}[theorem]{Problem}

\newcommand{\eps}{\varepsilon}
\newcommand{\F}{\mathbb{F}}
\newcommand{\E}{\mathbb{E}}
\newcommand{\C}{\mathbb{C}}
\newcommand{\rank}{\textrm{rank}}
\newcommand{\arank}{\textrm{arank}}
\newcommand{\srank}{\textrm{srank}}
\newcommand{\prank}{\textrm{prank}}
\newcommand{\bias}{\textrm{bias}}
\newcommand{\e}[1]{\chi\left(#1\right)}

\newcommand{\x}{\textbf{x}}
\newcommand{\y}{\textbf{y}}

\newcommand{\bc}{\textbf{b}}

\begin{document}

\begin{frontmatter}[classification=text]

\title{The Analytic Rank of Tensors and Its Applications} 

\author[sl]{Shachar Lovett\thanks{Supported by NSF award 1614023.}}

\begin{abstract}
The analytic rank of a tensor, first defined by Gowers and Wolf in the context of higher-order Fourier analysis, is defined to be the logarithm of the bias of the tensor. We prove that it is a subadditive measure of rank: that is, the analytic rank of the sum of two tensors is at most the sum of their individual analytic ranks.

This analytic property turns out to have surprising applications: (i) common roots of tensors are always positively correlated; and (ii) the slice rank and partition rank, which were defined recently in the resolution of the cap-set problem in Ramsey theory, can be replaced by the analytic rank.
\end{abstract}
\end{frontmatter}

\section{Introduction}

The main objects of study in this paper are tensors, or equivalently multilinear forms. Tensors have intimate relations to central problems in computer science and combinatorics. The complexity of matrix multiplication
is captured by the rank of the matrix multiplication tensor, see for example the survey \cite{gs005}. In arithmetic complexity, proving super-linear lower bounds for tensors
is related to proving lower bounds for arithmetic circuits and formulas, see for example \cite{chillara2016chasm} and the citations within. More relevant to the topic
of the current paper, recently defined notions of ranks were instrumental in the resolution of the cap-set problem in additive combinatorics
\cite{croot2017progression,ellenberg2017large,tao2016symmetric}, which is in itself intimately related to the problem of matrix multiplication~\cite{alon2013sunflowers,blasiak2017cap}.

The standard notion of tensor rank, as well as the more recent notions (called slice rank and partition rank, which will be formally defined shortly) are inherently combinatorial notions of rank.
The focus on this paper in on an \emph{analytic} notion of rank, which was first defined by Gowers and Wolf \cite{gowers2011linear-a} in the context of higher-order Fourier
analysis. The purpose is to (i) explore the power of this new notion of rank and (ii) connect it to the more well studied combinatorial notions. Our main results
can be informally stated as follows.

\begin{theorem}[Main results, informal]
For any tensor, the analytic rank lower bounds all the previously known combinatorial notions of rank (standard one, slice rank and partition rank). Moreover,
it can replace the role of the slice rank or partition rank in applications in Ramsey theory in product spaces.
\end{theorem}

The main reason why we find this theorem interesting is that the current techniques used to resolve the cap-set problem do not seem to extend to more general
problems of a similar flavour, except in very few cases. In general, the problems studied are Ramsey-type problems in product spaces, where the goal is
to upper bound the maximal size of a set without a particular sub-structure. We give two examples that illustrate this.

\begin{example}[$k$-AP free sets in $\mathbb{F}_p^n$]
\label{ex1}
Let $k \ge 3$, $p \ge k$ be a fixed prime, and $n$ be large. A $k$-AP (length $k$ arithmetic progression) in $\F_p^n$ is a set of the form $x,x+d,x+2d,\ldots,x+(k-1)d \in \F_p^n$ with $d \ne 0$.
The cap-set problem asks what is the maximal size of $A \subset \F_p^n$ which is $3$-AP free
(sometimes the term ``cap-set" is reserved for $p=3$, but we will abuse notation and extend it to general $p$).
Ellenberg and Gijswijt \cite{ellenberg2017large} proved that the answer is $O(c^n)$ for some $c<p$; that is,
$A$ has to be exponentially small. Previously, the best known bound was $O(p^n/n^{1+\eps})$ for some $\eps>0$ \cite{bateman2012new}.
However, these new techniques are seemingly unable to extend to bound the size of $4$-AP free sets, where the best known bound is
$O(p^n / n^{\eps})$ for some $\eps>0$ \cite{green2009new}.
\end{example}

\begin{example}[Erd\H{o}s-Szemer\'{e}di sunflower-free sets]
\label{ex2}
Consider families of subsets of $[n]=\{1,\ldots,n\}$. A $k$-sunflower is a set $x_1,\ldots,x_k \subset [n]$ where $x_i \cap x_j = x_1 \cap \ldots \cap x_k$ for all $i \ne j$.
Naslund and Sawin \cite{naslund2017upper} used techniques similar to these used to resolve the cap-set problem, to prove that the largest $3$-sunflower-free set $A \subset \{0,1\}^n$ has size $O(c^n)$
for some $c<2$. Again, $A$ is exponentially small. However, the techniques seem to fail to extend to bound the size of $4$-sunflower-free sets, where the best bounds are $2^{n-O(\sqrt{n})}$
(this bound follows from the Erd\"{o}s-Rado sunflower theorem \cite{erdos1960intersection} via standard reductions).
\end{example}

Thus, we see that while the new tensor-based techniques achieve amazing success on some problems, they are not robust in the sense that they do not generalize easily. One of the results of this paper is that the analytic rank can replace the role played by the slice rank or partition rank
in the current proofs, and in fact it is always a lower bound for these latter ranks. Thus, this raises an alternative approach to using tensor-rank based techniques to study these Ramsey problems.

\subsection{Tensors and tensor ranks}
We start by giving a formal definition of tensors and tensor ranks.

\paragraph{Tensors.}
Let $\F$ be a field, $V$ a finite dimensional linear space over $\F$, and $d \ge 1$. An order-$d$ tensor
(also called a $d$-linear form) is a multilinear map $T:V^d \to \F$. Equivalently, if $V$ is $n$-dimensional,
then we can identify $V \cong \F^n$, in which case
$$
T(x^1,\ldots,x^d) = \sum_{i_1,\ldots,i_d \in [n]} T_{i_1,\ldots,i_d} x^1_{i_1} \cdots x^{d}_{i_d}.
$$
Here we use the convention $[n]=\{1,\ldots,n\}$ and $x^i=(x^i_1,\ldots,x^i_n) \in \F^n$ for $i=1,\ldots,d$. The tensor $T$ is identified with the $d$-dimensional array
of its coefficients $(T_{i_1,\ldots,i_d}: i_1,\ldots,i_d \in [n])$.

\paragraph{Tensor ranks.}
There are several ``combinatorial" notions of tensor rank studied in the literature.
They all have the following form: the rank of $T$ is the minimal $r \ge 1$, such that $T$ can be factored as the sum of $r$ rank one tensors.
The only difference is what is considered to be a ``rank one tensor".

The most common definition, which is usually simply called ``rank", is that $T$ is rank one if it can be factored as
$$
T(x^1,\ldots,x^d) = T_1(x^1) T_2(x^2) \cdots T_d(x^d),
$$
where each $T_i$ is a an order-$1$ tensor (namely, a linear function). Recently, in the study of the cap-set problem and followup works,
two other definitions were introduced. A tensor $T$ has ``slice rank one" \cite{croot2017progression,ellenberg2017large,tao2016symmetric} if it can be factored as
$$
T(x^1,\ldots,x^d) = T_1(x^i) T_2(x^j: j \ne i)
$$
where $T_1$ is an order-$1$ tensor and $T_2$ is an order-$(d-1)$ tensor. A tensor $T$ has ``partition rank one" \cite{naslund2017partition} if it can be factored as
$$
T(x^1,\ldots,x^d) = T_1(x^i: i \in A) T_2(x^j: j \notin A)
$$
where $A \subset [d]$ is a set which satisfies $1 \le |A| \le d-1$, $T_1$ is an order-$|A|$ tensor, and $T_2$ is an order-$(d-|A|)$ tensor.

Let us denote the rank, slice rank, and partition rank of a tensor $T$ by $\rank(T), \srank(T), \prank(T)$, respectively. Then since rank one tensors are also slice rank one tensors,
and slice rank one tensors are also partition rank one tensors, we have:
$$
\prank(T) \le \srank(T) \le \rank(T).
$$

\subsection{The analytic rank}
A different notion of rank was introduced by Gowers and Wolf \cite{gowers2011linear-a} in the context of higher-order Fourier analysis.
Let $\F$ be a finite field, and let $\chi:\F \to \C$ be a nontrivial additive character (for example, if $\F=\F_p$ is a prime field,
we can take $\chi(x) = \exp(2 \pi i x / p)$). Given a function $F:X \to \F$, its \emph{bias} is
$$
\bias(T) := \E_{x \in X}[\e{F(x)}].
$$
In particular, let $T:V^d \to \F$ be an order-$d$ tensor. The bias of $T$ is always real and in $(0,1]$. To see that, define $T(\cdot,x^2,\ldots,x^d)$ to be the order-$1$ tensor on $x^1$ given a fixing of $x^2,\ldots,x^d$. Then
\begin{equation}\label{eq:bias}
\bias(T) = \E_{x^2,\ldots,x^d \in V} \left[ \E_{x^1 \in V} [\e{T(x^1,\ldots,x^d)}] \right] = \Pr_{x^2,\ldots,x^d \in V} \left[ T(\cdot,x^2,\ldots,x^d) \equiv 0 \right].
\end{equation}
This is since for an order-$1$ tensor (namely, a linear form), its bias is $1$ if it is identically zero, and is $0$ otherwise.
The \emph{analytic rank} of $T$ is defined to be
$$
\arank(T) := -\log_{|\F|} \bias(T).
$$
As $\bias(T) \in (0,1]$ we have that $\arank(T) \ge 0$. The following example might help shed some light on the definition. It shows
that in the case of order-$2$ tensors (namely, bilinear forms, corresponding to matrices), the analytic rank is equivalent to the standard notion of rank.

\begin{example}
Consider the order-$2$ tensor $T:(\F^n)^2 \to \F$ defined as $T(x,y)=\sum_{i=1}^r x_i y_i$. Then by \Cref{eq:bias},
$$
\bias(T) = \Pr_{y \in \F^n}[y_1=\ldots=y_r=0] = |\F|^{-r}.
$$
Hence $\arank(T) = r$ which coincides with the usual notion of rank for bilinear forms.
\end{example}

Gowers and Wolf \cite{gowers2011linear-a} proved that the analytic rank is approximately sub-additive,
in the sense that $\arank(T+S) \le 2^d \left( \arank(T)+\arank(S) \right)$. We show that the analytic rank is in fact sub-additive.
The fact that we do not lose any constant factor is crucial in the applications.

\begin{theorem}\label{thm:subadditive}
Let $T,S: V^d \to \F$ be order-$d$ tensors. Then
$$
\arank(T+S) \le \arank(T) + \arank(S).
$$
\end{theorem}

We note that the bound in \Cref{thm:subadditive} is tight: if $T,S$ are defined over disjoint variables, then it is easy to verify that $\bias(T+S)=\bias(T)\bias(S)$ and hence $\arank(T+S)=\arank(T)+\arank(S)$.

\subsection{Applications}
\Cref{thm:subadditive} has some surprising applications, which we describe next.

\paragraph{Common roots of tensors are positively correlated.}
We show that the common roots of order-$d$ tensors on a common input are always positively correlated.

\begin{claim}
\label{claim:pos_cor}
Let $T_1,\ldots,T_m,S_1,\ldots,S_n:V^d \to \F$ be order-$d$ tensors. Then
\begin{align}
\label{eq:pos_cor}
&\Pr_{\x \in V^d}[T_1(\x)=\ldots=T_m(\x)=S_1(\x)=\ldots=S_n(\x)=0] \ge \\
&\Pr_{\x \in V^d}[T_1(\x)=\ldots=T_m(\x)=0] \cdot \Pr_{\x \in V^d}[S_{1}(\x)=\ldots=S_n(\x)=0] \nonumber.
\end{align}
\end{claim}

\begin{proof}
Define two order-$(d+1)$ tensors as follows:
$$
T(x^0,x^1,\ldots,x^d) = \sum_{i=1}^m x^0_i T_i(x^1,\ldots,x^d), \qquad S(x^0,x^1,\ldots,x^d) = \sum_{i=m+1}^{m+n} x^0_i S_i(x^1,\ldots,x^d).
$$
\Cref{eq:bias} gives that the LHS of \Cref{eq:pos_cor} is equal to $\bias(T+S)$, whereas the RHS is equal to $\bias(T)\bias(S)$.
The claim then follows from \Cref{thm:subadditive}.
\end{proof}

\paragraph{The analytic rank can replace the partition rank.}
The motivation behind the introduction of the slice rank and the partition rank, was to study the cap-set problem, and more generally
Ramsey problems in product spaces. Works in this space include
\cite{croot2017progression,ellenberg2017large,naslund2017partition,blasiak2017cap,kleinberg2016nearly,kleinberg2016growth,naslund2017upper}.

In all these problems, a certain tensor $T:(\F^n)^d \to \F$ is defined which captures the problem structure. An \emph{independent set} in $T$ is a subset $A \subset [n]$
that satisfies
$$
\forall i_1,\ldots,i_d \in A: \qquad T_{i_1,\ldots,i_d} \ne 0 \quad \Leftrightarrow \quad i_1=\ldots=i_d.
$$
The goal is to upper bound the size of the largest
independent set in $T$. The proofs combine the following two properties:
\begin{enumerate}
\item[(i)] The slice rank, or partition rank, of the specific tensor $T$ studied is low. This is usually an ad-hoc argument, which relies
on the specific definition of $T$.
\item[(ii)] If $T$ contains an independent set $A$ then its partition rank (and hence also its slice rank) is at least $|A|$.
\end{enumerate}
This allows to upper bound the size of the maximal independent set in $T$. We show that the analytic rank can be used instead of the slice rank, or the partition rank, and that it gives comparable bounds. This raises the possibility of proving bounds on the analytic rank directly, which may circumvent some of the challenges
in extending the current line of work to other Ramsey problems (see \Cref{ex1} and \Cref{ex2} and the discussion that follows).

\begin{theorem}
\label{thm:properties}
Let $T:V^d \to \F$ be an order-$d$ tensor. Then
\begin{enumerate}
\item[(i)] $\arank(T) \le \prank(T)$.
\item[(ii)] If $T$ contains an independent set $A$ then $\arank(T) \ge c |A|$.
\end{enumerate}
Here $c=c(d,|\F|)$ satisfies that $c \ge 2^{-d}$ always and $c \ge 1-\frac{\log(d-1)}{\log |\F|}$ which is better for large $\F$.
\end{theorem}

We make two remarks. First, claim (i) was proved independently by Kazhdan and Ziegler \cite[Lemma 2.2]{kazhdan2018approximate}. The reader can verify that the unspecified constant $C_{L,d}$ in their Lemma 2.2 equals $|k|^{-L}$, from which the claim follows. Next, in claim (ii),
the fact that we do not obtain $c=1$ is not important for the applications, as typically $d$ is a small constant (say $3$ or $4$) and the goal is to prove bounds on $|A|$ which are exponential in $n=\dim(V)$, which is assumed to be large.

\subsection{Is the analytic rank really better than the partition rank?}
Given \Cref{thm:properties}, a natural question arises: is the analytic rank ``better" than the partition rank? namely,
are there tensors $T$ where $\arank(T) \ll \prank(T)$? This question is intimately related to the line of work known
as ``bias implies low rank" in higher-order Fourier analysis \cite{green2009distribution,kaufman2008worst,haramaty2010structure,bhowmick2018list,bhowmick2015bias}. Re-interpreting these
results in the language of analytic rank vs partition rank, the known results until very recently were:
\begin{enumerate}
\item[(i)] If $T$ is an order-$d$ tensor with $\arank(T) \le r$, then $\prank(T) \le f(r,d)$, where $f$ has an Ackerman-type dependence on its parameters \cite{bhowmick2015bias}. Note that $f$ does not depend on the underlying field $\F$ or the dimension of the tensor $n$.
\item[(ii)] For $d=3,4$, better quantitative bounds are known: $f(r,3) = O(r^4)$ and $f(r,4) = \exp(O(r))$ \cite{haramaty2010structure}.
\item[(iii)] Very recently, in ground-breaking independent works, Janzer \cite{janzer2019polynomial} and Mili{\'c}evi{\'c} \cite{milicevic2019polynomial} improved the bounds to polynomial for all $d$. Namely,
$f(r,d) = c_d r^{c_d}$ where $c_d$ is a constant which depends only on $d$.
\end{enumerate}

However, there are no examples known where the gap between the analytic rank and partition rank is more than a constant.
The best separation we know of is for the \emph{identity tensor}.

\begin{example}[Identity tensor]
Let $I:(\F_p^n)^d \to \F_p$ be the identity tensor, defined as
$$
I(x^1,\ldots,x^d) = \sum_{i=1}^n \prod_{j=1}^d x^j_i.
$$
Naslund \cite{naslund2017partition} proved that
$I$ has maximal partition rank, namely $\prank(I)=n$. On the other hand, the calculation in the proof of \Cref{thm:properties} shows that
$\arank(I) = cn$ where $c=c(d,p)$ is the constant given in \Cref{thm:properties}.
\end{example}

We refer the reader also to \cite{bhrushundi2018multilinear}, which
analyzes the relation between the analytic rank and rank of tensors,
in the context of proving arithmetic circuit lower bounds.
This leads to the following natural question.

\begin{problem}
Is it true that for any order-$d$ tensor $T$ it holds that $\prank(T) \le c_d \arank(T)$, where $c_d$ is a constant which depends
only on $d$?
\end{problem}

We conclude with another interesting problem, relating to the scope of definition of the analytic rank.

\begin{problem}
Can the notion of analytic rank be extended beyond finite fields? For example, for tensors defined over $\mathbb{R}$ or $\mathbb{C}$?
\end{problem}

We note that Gowers and Wolf in \cite{gowers2011linear-b} also defined analytic rank for functions over $\mathbb{Z}_N$, but the treatment there does not seem related to the problems studied in this paper.

\paragraph{Organization.} \Cref{thm:subadditive} is proved in \Cref{sec:subadditive}, and \Cref{thm:properties} is proved in \Cref{sec:properties}.

\section{Proof of \Cref{thm:subadditive}}
\label{sec:subadditive}
We prove \Cref{thm:subadditive} in this section.
It suffices to prove that for any two order-$d$ tensors $T,S:V^d \to \F$ it holds that
\begin{equation}
\label{eq:bias_T_plus_S}
\bias(T+S) \ge \bias(T) \bias(S).
\end{equation}
We first introduce some notation. Let $\x=(x^1,\ldots,x^d), \y=(y^1,\ldots,y^d) \in V^d$.
For $I \subseteq [d]$ define $I^c = [d] \setminus I$ and shorthand $\x^I := (x^i: i \in I)$.
Define $T_{I}(\x,\y)$ as
$$
T_{I}(\x,\y) := T_{I}(\x^I, \y^{I^c}) = T(z^1,\ldots,z^d), \text{ where }
z^i =
\begin{cases}
x^i & \text{if } i \in I\\
y^i & \text{if } i \notin I
\end{cases}\;.
$$
That is, $T_I$ is the tensor $T$ evaluated over $\x^I, \y^{I^c}$.
Observe that $T(\x+\y)$ decomposes as the sum
\begin{equation}
\label{eq:T_sum}
T(\x+\y) = \sum_{I \subseteq [d]} T_I(\x^I, \y^{I^c}).
\end{equation}
Express $\bias(T) \bias(S)$ as
$$
\bias(T) \bias(S) = \bias \left( T(\x) + S(\y) \right) = \bias \left( T(\x) + S(\x+\y) \right).
$$
Here, we used the fact that the joint distributions of $(\x,\y)$ and $(\x,\x+\y)$ are identical. Next, decompose $S(\x+\y)$ using \Cref{eq:T_sum} as
\begin{align*}
\bias(T) \bias(S)
&= \bias \left( T(\x) + \sum_{I \subseteq [d]} S_I(\x^I, \y^{I^c}) \right)\\
&= \bias \left( (T+S)(\x) + \sum_{I \subsetneq [d]} S_I(\x^I, \y^{I^c}) \right).
\end{align*}
Fix a choice of $\y=\bc \in V^d$ so that
$$
\bias(T) \bias(S) \leq
\left|\bias \left( (T+S)(\x) + \sum_{I \subsetneq [d]} S_I(\x^I, \bc^{I^c}) \right)\right|.
$$
Observe that $S_{I}(\x^I, \bc^{I^c})$ is an order-$|I|$ tensor of the inputs $\x^I$.
The proof of \Cref{thm:subadditive} follows from the following lemma,
applied to $R_{[d]}(\x)=(T+S)(\x)$ and $R_I(\x^I) = S_I(\x^I,\bc^{I^c})$ for $I \subsetneq [d]$.

\begin{lemma}
\label{lemma:bias}
For each $I \subseteq [d]$, let $R_I:V^I \to \F$ be an order-$|I|$ tensor. Consider the function
$$
R(\x) = \sum_{I \subseteq [d]} R_I(\x^I).
$$
Then
$$
|\bias(R)| \le \bias(R_{[d]}).
$$
\end{lemma}

In order to prove \Cref{lemma:bias}, we first prove the following claim.

\begin{claim}
\label{claim:w}
Let $W_0,\ldots,W_n:\F^m \to \F$ be functions. Consider functions $A,B:\F^n \times \F^m \to \F$ defined as follows:
$$
A(x,y)=\sum_{i=1}^n x_i W_i(y), \qquad B(x,y)=A(x,y) + W_0(y).
$$
Then
$$
|\bias(B)| \le \bias(A).
$$
\end{claim}

\begin{proof}
We have
$$
\bias(B) = \E_{x,y} \left[ \e{B(x,y)} \right] = \E_{y} \left[1_{W_1(y)=\ldots=W_n(y)=0} \cdot \e{W_0(y)}\right].
$$
Hence
$$
|\bias(B)| \le \E_{y} \left[1_{W_1(y)=\ldots=W_n(y)=0}\right] = \bias(A).
$$
\end{proof}

\begin{proof}[Proof of \Cref{lemma:bias}]
Fix $i \in [d]$ and decompose $R(\x)$ as
$$
R(\x) = \sum_{I \subseteq [d], i \in I} R_I(\x^I) + \sum_{I \subseteq [d], i \notin I} R_I(\x^I).
$$
Setting $x=x^i$ and $y=\x^{[d] \setminus \{i\}}$, the first sum has the form $\sum x_i W_i(y)$, and the second sum which
does not depend on $x$ is $W_0(y)$. \Cref{claim:w} then gives that
$$
|\bias(R)| \le \bias \left(\sum_{I \subseteq [d], i \in I} R_I(\x^I) \right).
$$
Applying this iteratively for $i=1,\ldots,d$ completes the proof.
\end{proof}

\section{Proof of \Cref{thm:properties}}
\label{sec:properties}
We prove \Cref{thm:properties} in this section. We break it as a series of claims.
We first show that the analytic rank is at most the partition rank.

\begin{claim}
Let $T:V^d \to \F$ be an order-$d$ tensor. Then $\arank(T) \le \prank(T)$.
\end{claim}

\begin{proof}
Given \Cref{thm:subadditive}, it suffices to prove the claim for tensors $T$ of partition rank one.
Assume that $T:V^d \to \F$ factors as
$$
T(\x) = T_1(\x^A) T_2(\x^B)
$$
where $A \cup B$ is a partition of $[d]$, $|A|,|B| \ge 1$.
We will show that $\bias(T) \ge |\F|^{-1}$ which implies that $\arank(T) \le 1$.
For $a,b \in \F$ define the function
$$
F_{a,b}(\x) := \left( T_1(\x^A) + a \right) \left( T_2(\x^B) + b \right).
$$
\Cref{lemma:bias} gives that
$$
|\bias(F_{a,b})| \le \bias(T).
$$
On the other hand, if we let $a,b \in \F$ be chosen uniformly, then
\begin{align*}
\E_{a,b \in \F}[\bias(F_{a,b})] &= \E_{a,b \in \F, \x \in V^d} \left[ \e{(T_1(\x^A)+a)(T_2(\x^B)+b)} \right] \\
&= \E_{a,b \in \F} \left[ \e{ab} \right] = \Pr_{b \in \F}[b=0] = |\F|^{-1}.
\end{align*}
It follows that $\bias(T) \ge |\F|^{-1}$, as claimed.
\end{proof}

Next, we show that the analytic rank cannot increase in a restriction of a tensor to a subspace.

\begin{claim}
\label{claim:restrict}
Let $T:V^d \to \F$ be an order-$d$ tensor, let $U \subset V$ be a subspace and consider the restricted tensor $T|_U: U^d \to \F$. Then $\arank(T|_U) \le \arank(T)$.
\end{claim}

\begin{proof}
Let $W \subset V$ be a subspace so that $U \oplus W = V$. Each $x \in V$ can be written uniquely
as $x=u+w$ with $u \in U, w \in W$. We have
$$
\bias(T) = \E_{x^1,\ldots,x^d \in V} \left[ \e{T(x^1,\ldots,x^d)} \right]
= \E_{u^1,\ldots,u^d \in U, w^1,\ldots,w^d \in W} \left[ \e{T(u^1+w^1,\ldots,u^d+w^d)} \right].
$$
Consider any fixing of $w^1,\ldots,w^d \in W$. Then
$$
T(u^1+w^1,\ldots,u^d+w^d) = \sum_{I \subseteq [d]} T_I (u^I, w^{I^c})
$$
where $T_I(u^I,w^{I^c})=T(z_1,\ldots,z_d)$ where $z_i = u_i$ if $i \in I$ and $z_i=w_i$ if $i \notin I$.
\Cref{lemma:bias} gives that
$$
\left| \E_{u^1,\ldots,u^d \in U} \left[ \e{T(u^1+w^1,\ldots,u^d+w^d)} \right] \right| \le
\E_{u^1,\ldots,u^d \in U} \left[ \e{T(u^1,\ldots,u^d)} \right]  = \bias(T|_U).
$$
The claim follows by averaging over $w^1,\ldots,w^d \in W$.
\end{proof}

Finally, we show that if a tensor contains an independent set $A$ then its analytic rank is at least linear in $|A|$.

\begin{claim}
Let $T:(\F^n)^d \to \F$ be an order-$d$ tensor.
Assume that $A \subseteq [n]$ is an independent set in $T$. Then $\arank(T) \ge c |A|$ where
$c=c(d,|\F|)$ satisfies
$c \ge 2^{-d}$ and $c \ge 1-\frac{\log(d-1)}{\log |\F|}$.
\end{claim}

\begin{proof}
Let $S:(\F^A)^d \to \F$ be the restriction of $T$ to $\F^A$. By \Cref{claim:restrict} we have $\arank(T) \ge \arank(S)$. We have
$$
\bias(S) = \E_{x^1,\ldots,x^d \in \F^A} \left[ \e{\sum_{i \in A} c_i \prod_{j=1}^d x^j_i} \right]
$$
where $c_i \ne 0$ for $i \in A$. \Cref{eq:bias} gives that
$$
\bias(S) = \Pr_{x^2,\ldots,x^d \in \F^A}\left[\prod_{j=2}^d x^j_i=0 \quad \forall i \in A \right].
$$
A simple calculation then gives
$$
\bias(S) = \left(1 - \left(1-\frac{1}{|\F|}\right)^{d-1} \right)^{|A|}.
$$
Define $c(d,y) = -\log_{y} \left(1 - \left(1-\frac{1}{y}\right)^{d-1}\right)$ so that $\arank(S) = c(d,|\F|) \cdot |A|$.

A convexity argument shows that for $y \ge 2$, $c(d,y) \ge c(d,2) = -\log_2 (1-2^{-(d-1)}) \ge 2^{-(d-1)}$. Next, if
we assume $y \ge d$ (otherwise the second bound on $c$ is trivial)
then $c(d,y) \ge -\log_{y}  \left( \frac{d-1}{y} \right) = 1 - \frac{\log(d-1)}{\log y}$.
\end{proof}

\bibliographystyle{amsplain}
\providecommand{\bysame}{\leavevmode\hbox to3em{\hrulefill}\thinspace}
\providecommand{\MR}{\relax\ifhmode\unskip\space\fi MR }
\providecommand{\MRhref}[2]{%
  \href{http://www.ams.org/mathscinet-getitem?mr=#1}{#2}
}
\providecommand{\href}[2]{#2}


\begin{dajauthors}
\begin{authorinfo}[sl]
  Shachar Lovett\\
  Computer Science and Engineering\\
  University of California, San Diego\\
  slovett\imageat{}ucsd\imagedot{}edu\\
  \url{http://cseweb.ucsd.edu/~slovett}
\end{authorinfo}
\end{dajauthors}

\end{document}